\documentclass[12pt]{amsart}
\usepackage{amsrefs,amssymb,amsmath,amssymb,amsthm}
\usepackage[ocgcolorlinks,linktoc=all]{hyperref}
\usepackage{lipsum}
\usepackage{esint}
\usepackage{enumerate}

\hypersetup{citecolor=blue,linkcolor=magenta}
\newtheorem{theorem}{Theorem}[section]

\newtheorem{lemma}[theorem]{Lemma}

\newtheorem{assumption}[theorem]{Assumption}

\theoremstyle{definition}
\newtheorem{definition}[theorem]{Definition}

\theoremstyle{remark}
\newtheorem{remark}[theorem]{Remark}
\numberwithin{equation}{section}

\newcommand{\OP}{\Lambda_p^{\varphi}}
\newcommand{\Op}{\Lambda_p}
\newcommand{\Sp}{\mathbb{S}^{n-1}}

\begin{document}

\title[Iterations of curvature images]
 {Iterations of curvature images}
\author[M. N. Ivaki]{Mohammad N. Ivaki}
\address{Department of Mathematics, University of Toronto, Ontario,
M5S 2E4, Canada}
\email{m.ivaki@utoronto.ca}
\dedicatory{}
\subjclass[2010]{}
\begin{abstract}
We study the iterations of a class of curvature image operators $\Lambda_p^{\varphi}$ introduced by the author in (J. Funct. Anal. 271 (2016) 2133--2165). The fixed points of these operators are the solutions of the $L_p$ Minkowski problems with the positive continuous prescribed data $\varphi$. One of our results states that if $p\in (-n,1)$ and $\varphi$ is even, or if $p\in (-n,-n+1]$, then the iterations of these operators applied to suitable convex bodies sequentially converge in the Hausdorff distance to fixed points.
\end{abstract}

\maketitle
\section{Introduction}
The setting here is $n$-dimensional Euclidean space. Let $\varphi\in C(\Sp)$ be a positive continuous function defined on the unit sphere. Suppose either $\varphi$ is even (i.e., it takes the same value at antipodal points) and $p\in (-n,1),$ or $p\in (-n,-n+1]$. Using an iteration method, we show that there exists a convex body $K$ with support function $h_K$ and curvature function $f_K$ such that
\begin{align}\label{lp problem}
\varphi h_K^{1-p}f_K=\textrm{const}.
\end{align}
While the existence of solutions to \eqref{lp problem} has been known in this range of $p$ since the work of Chou--Wang \cite{CW}, we use a notion of generalized curvature image to add a novel existence method to the literature on the $L_p$ Brunn-Minkowski theory.

Let us briefly recall the origin and the historical context of (\ref{lp problem}). For any $x$ on the boundary of a convex body $K,$ $\nu_K(x)$ is the set of all unit exterior normal vectors at $u.$ The surface area measure of $K$, $S_K$, is a Borel measure on the unit sphere defined by
\[S_K(\omega)=\mathcal{H}^n(\nu_K^{-1}(\omega))\quad\textrm{for all Borel sets}~ \omega~ \textrm{of}~\Sp.\]
Here, $\mathcal{H}^n$ denotes the $n$-dimensional Hausdorff measure. If $K$ has a positive continuous curvature function, then $dS_K=f_Kd\sigma,$ where $\sigma$ is the spherical Lebesgue measure.

The classical Minkowski problem is one of the corner stones of the Brunn-Minkowski theory. It asks what are the necessary and sufficient conditions on a Borel measure $\mu$ on $\Sp$ in order to be the surface area measure of a convex body. The complete solution to this problem was found by Minkowski, Aleksandrov and Fenchel and Jessen (see, e.g., Schneider \cite{Schneider}): A Borel measure $\mu$ whose support is not contained in a closed hemisphere is the surface area measure of a convex body if and only if
\[\int_{\Sp} ud\mu(u)=o.\]
Moreover, the solution is unique up to translations.

The $L_p$ Minkowski asks what are the necessary and sufficient conditions on a Borel measure $\mu$ on $\Sp$, such that there exists a convex body $K$ with support function $h_K$, so that
\[h_K^{1-p}dS_K=\gamma d\mu \quad\hbox{for some constant}~\gamma>0.\]
This problem for $p>1$ was put forward by Lutwak \cite{Lutwak1993} almost a century after Minkowski's original work and stems from the $L_p$ linear combination of convex bodies. See \cites{Schneider,BBCY,BBC,CW,Stancu02,Stancu03,LO,BLYZa,JLW,BT17,CHLL} regarding the $L_p$ Minkowski problem and Lutwak \textit{et al}. \cite{LYZapp} for an application.

To motivate our iteration scheme, let us briefly recall a few observations from Lutwak \cite{Lutwak86}. Suppose $K$ has its Santal\'{o} point at the origin. Then by Minkowski's existence theorem (see, e.g., \cite{Busemann}*{pp. 60--67}), there exists a convex body $\Lambda K$, uniquely determined up to translations, whose curvature function is given by
\[f_{\Lambda K}=\frac{V(K)}{V(K^{\ast})}\frac{1}{h_K^{n+1}}.\]
Here, $K^{\ast}$ is the polar body and $V(\cdot)$ is the $n$-dimensional Lebesgue measure.
We always choose $\Lambda K$ such that its Santal\'{o} point is at the origin. The curvature image operator $\Lambda$ was introduced by Petty \cite{Petty}. See \cite{Schneider}*{Section 10.5} for the importance of the curvature image in affine differential geometry.

Write $\Omega(K)$ for the affine surface area of $K$ (Definition \ref{def} with $\varphi\equiv 1, p=-n$). By a straightforward calculation,
\[\Omega(\Lambda K)^{n+1}=n^{n+1}V(K)^{n}V(K^{\ast}).\]
On the other hand, for any convex body $L$ with the origin in its interior by the H\"{o}lder inequality, we have
\[\Omega(L)^{n+1}\leq n^{n+1}V(L)^nV(L^{\ast}).\]
Hence, using this inequality for $L=K$ and $L=\Lambda K$, we see
\begin{align*}
\Omega(\Lambda K)&\geq \Omega(K)\\
V(K)V(K^{\ast})&\leq \left(\frac{V(\Lambda K)}{V(K)}\right)^{n-1}V(\Lambda K)V((\Lambda K)^{\ast}).
\end{align*}
By Minkowski's mixed volume inequality, we have \[V(\Lambda K)\leq V(K) .\]
Moreover, using the affine isoperimetric inequality,
\[V(\Lambda K)^{n-1}\geq \frac{\Omega(\Lambda K)^{n+1}}{n^{n+1}V(B)^2}\geq \frac{\Omega(K)^{n+1}}{n^{n+1}V(B)^2},\]
where $B$ denotes the unit ball. Therefore, we arrive at
\[\left(\frac{\Omega(K)^{n+1}}{n^{n+1}V(B)^2}\right)^{\frac{1}{n-1}}\leq V(\Lambda K)\leq V(K).\]
Let us put $\Lambda^i K:=\underbrace{\Lambda\cdots\Lambda}\limits_{i~\textrm{times}} K.$
By induction, we obtain
\begin{align*}
V(\Lambda ^{i-1}K)V((\Lambda ^{i-1}K)^{\ast})&\leq \left(\frac{V(\Lambda^i K)}{V(\Lambda^{i-1}K)}\right)^{n-1} V(\Lambda ^{i}K)V((\Lambda ^{i}K)^{\ast}),\\
 \left(\frac{\Omega(K)^{n+1}}{n^{n+1}V(B)^2}\right)^{\frac{1}{n-1}}&\leq V(\Lambda^iK)\leq V(K).
\end{align*}

To sum up these observations, we have seen the curvature image under the operator $\Lambda$ strictly increases (unless it is applied to an origin-centered ellipsoid; see Marini--De Philippis \cite{MP} regarding the fact that the only solutions of $\Lambda K=K$ are origin-centered ellipsoids) the volume product functional, while $\{V(\Lambda^i K)\}_i$ is uniformly bounded above and below away from zero.

The previous observations motivate us to seek a curvature image operator $\Lambda^{\varphi}_{p}$ (see Definition \ref{def curv op}) that satisfies the following three rules.
\begin{enumerate}
  \item The fixed points, $\Lambda^{\varphi}_{p}L=L$, are solutions of (\ref{lp problem}).
  \item The curvature image under $\Lambda^{\varphi}_{p}$ \emph{strictly} increases a ``suitable" functional, unless $\Lambda^{\varphi}_{p}$ is applied to a solution of (\ref{lp problem}).
  \item There are uniform lower and upper bounds on the volume after applying any number of iteration.
\end{enumerate}
Put $(\OP)^i K:=\underbrace{\OP\cdots\OP}\limits_{i~\textrm{times}} K.$ When $\varphi\equiv 1$, we use $\Op$ in place of $\OP.$
\begin{theorem}\label{main theorem} The following statements hold:
\begin{enumerate}
  \item  suppose either
  \begin{itemize}
  \item $-n<p<1$, $\varphi\in  C(\Sp)$ is positive and even, and $K$ is origin-symmetric, or
  \item $-n< p\leq -n+1$, $\varphi\in  C(\Sp)$ is positive, $K$ contains the origin in its interior and
  \[\int_{\Sp}\frac{u}{(\varphi h_K^{1-p})(u)}d\sigma=o.\]
\end{itemize}
Then a subsequence of iterations $\{(\OP)^{i}K\}_i$ converges in the Hausdorff distance, as $i\to\infty$, to a convex body $L$ such that
\[\varphi h_L^{1-p}f_L=\textrm{const}.\]
  \item If  $-n<p<1$, and $K$ contains the origin in its interior and
  \[\int_{\Sp}uh_K^{p-1}(u)d\sigma=o,\] then $\{\Lambda_p^{i}K\}_i$ converges in the Hausdorff distance, as $i\to\infty$, to an origin-centered ball.
  \item If $p=-n$ and $K$ has its Santal\'{o} point at the origin, then there exists a sequence of volume-preserving transformations $\ell_i$, such that $\{\ell_i\Lambda^{i}K\}_i$ converges in the Hausdorff distance, as $i\to\infty$, to an origin-centered ball.
\end{enumerate}
\end{theorem}
\begin{remark}
Each convex body after a suitable translation satisfies the required integral condition in the theorem (see, e.g., \cite{Ivaki2016}*{Lemma 3.1}).
%Moreover, it is clear that the use of volume-preserving transformations to obtain a limit is essential (take $K$ to be an ellipsoid). This is in contrast to the behavior of the solutions of the affine normal flow \cite{Andrews96} (which is a certain infinitesimal iteration procedure; see \cite{Andrews00}) where such transformations are not needed.
\end{remark}
Iterations methods in convex geometry were previously applied in \cites{FNRZ,ZS} and as smoothing tools in \cites{Ivaki2017,Ivaki2018} to prove local uniqueness of fixed points of a certain class of operators. Also to deduce the asymptotic behavior of a class of curvature flows in \cites{Ivaki2016,Ivakitran,BIS} and to prove a stability version of the Blaschke--Santal\'{o} inequality in the plane \cite{Ivaki2015} we used some properties of the curvature image operators. We mention that the unique convex body of maximal affine perimeter contained in a given two-dimensional convex body is (up to translations) a curvature image body (see B\'{a}r\'{a}ny \cites{Barany97,Barany06}). Moreover, Schneider \cite{Schneider14} proved that in any dimension a curvature image body uniquely possess the maximal affine surface area among all convex bodies contained in it.
\section*{Acknowledgment}
MI has been supported by a Jerrold E. Marsden postdoctoral fellowship from the Fields Institute. The author would like to thank the referee for suggestions that led to improvement of this article.
\section{Background and notation}
A compact convex set with non-empty interior is called a convex body.
The set of convex bodies is denoted by $\mathcal{K}.$ Write $\mathcal{K}_o,\mathcal{K}_e$, respectively, for the set of convex bodies containing the origin in their interiors and the origin-symmetric convex bodies. We write $ C^+(\Sp)$ for the set of positive continuous functions and $ C_e^+(\Sp)$ for the set of positive continuous even functions on the unit sphere.

The support function of a convex body $K$ is defined as
\[h_K(u)=\max_{x\in K}\langle x,u\rangle. \]
For a convex body $K$ with the origin $o$ in its interior, the polar body $K^{\ast}$ is defined by
\[K^{\ast}=\{y: \langle x,y \rangle\leq 1~\forall x\in K\}.\]
For $x\in \operatorname{int}K$, we set $K^x=(K-x)^{\ast}$. The Santal\'{o} point of $K$, denoted by $s$, is the unique point in $\operatorname{int} K$ such that
\[V(K^s)\leq V(K^x)\quad \forall x\in \operatorname{int}K.\]
Moreover, the Blaschke--Santal\'{o} inequality states that
\[V(K)V(K^s)\leq V(B)^2.\]
with equality if and only if $K$ is an ellipsoid.

Let $K,L$ be two convex bodies and $0<a<\infty.$ The Minkowski sum $K+aL$ is defined by $h_{K+aL}=h_K+ah_L$ and the mixed volume of $K,L$ is defined by
\[V_1(K,L)=\frac{1}{n}\lim_{a\to 0}\frac{V(K+aL)-V(K)}{a}.\]
Corresponding to each $K,$ there is a unique Borel measure $S_K$ on the unit sphere such that
\[V_1(K,L)=\frac{1}{n}\int_{\Sp}h_LdS_K\quad \textrm{for any convex body}~L.\]
If the boundary of $K,$ $\partial K$, is $C^2$-smooth and strictly convex , then $S_K$ is absolutely continuous with respect to spherical Lebesgue measure $\sigma$ and $dS_K/d\sigma$ is the reciprocal Gauss curvature.

The Minkowski mixed volume inequality states that
\[V_1(K,L)^n\geq V(K)^{n-1}V(L),\]
and equality holds if and only if $K$ and $L$ are homothetic.

We say $K$ has a positive continuous curvature function $f_K$ if
\[V_1(K,L)=\frac{1}{n}\int_{\Sp}h_Lf_Kd\sigma\quad\textrm{for any convex body}~L.\]

Write $\mathcal{F}$ for set of convex bodies with positive continuous curvature functions, and put \[\mathcal{F}_o=\mathcal{K}_o\cap \mathcal{F},\quad\mathcal{F}_e=\mathcal{K}_e\cap \mathcal{F}.\]
\section{Curvature image operators}
\begin{assumption}\label{assumption}Suppose one of the following cases occurs.
\begin{enumerate}
  \item $-n<p\neq 1<\infty$, $\varphi\in  C_e^+(\Sp)$ and $K\in \mathcal{K}_e.$
  \item $-n< p\leq -n+1$, $\varphi\in  C^+(\Sp)$, $K\in \mathcal{K}_o$ and
  \[\int_{\Sp}\frac{u}{(\varphi h_K^{1-p})(u)}d\sigma=o.\]
  \item $-n\leq p\neq 1< \infty$, $K\in \mathcal{K}_o$ and
  \[\int_{\Sp}uh_K^{p-1}(u)d\sigma=o.\]
\end{enumerate}
\end{assumption}
\begin{definition}\label{def curv op} Under the Assumption \ref{assumption}, the curvature image $\Lambda_p^{\varphi} K$ of $K$ is defined as the unique convex body whose curvature function is
\begin{align}\label{curve image}
f_{\Lambda_p^{\varphi} K}=\frac{V(K)}{\frac{1}{n}\int_{\Sp}\frac{h^p_K}{\varphi}d\sigma}\frac{h_K^{p-1}}{\varphi}
\end{align}
and its support function satisfies
\begin{align}\label{curve image-1}
\int_{\Sp} \frac{u}{(\varphi h_{\OP K}^{1-p})(u)}d\sigma=o.
\end{align}
\end{definition}
\begin{remark}
By Minkowski's existence theorem and \cite{Ivaki2016}*{Lemma 3.1}, there \emph{exists a unique} convex body that satisfies (\ref{curve image}) and (\ref{curve image-1}). The integral condition (\ref{curve image-1}) for $p=-n$ and $\varphi\equiv 1$ says the curvature image has its Santal\'{o} point at the origin.
\end{remark}
In view of $V_1(\Lambda_p^{\varphi} K,K)=V(K)$ and Minkowski's mixed volume inequality, we have
\begin{align}\label{key}
V(K)\geq V(\Lambda_p^{\varphi} K).
\end{align}
Moreover, equality holds if and only if $\OP K=K.$
\begin{definition}\label{def} Suppose $\varphi\in  C^+(\Sp)$. For $K\in \mathcal{K}_o$, we define
\[
\mathcal{A}_p^{\varphi}(K)=\begin{cases}
V(K)\left(\int_{\Sp}\frac{h_K^p}{\varphi}d\sigma\right)^{-\frac{n}{p}}, &0\neq p\in [-n,\infty),\\
V(K)\exp\left(\frac{\int_{\Sp}-\frac{1}{\varphi}\log h_Kd\sigma}{\frac{1}{n}\int_{\Sp}\frac{1}{\varphi}d\sigma}\right), & p=0.
\end{cases}
\]
Let $K\in \mathcal{F}$. Define
\begin{itemize}
  \item  for $p\in [-n,\infty)\setminus\{0,1\}:$
  \[\mathcal{B}_p^{\varphi}(K)=V(K)^{1-n}\left(\int_{\Sp}\varphi^{\frac{1}{p-1}}f_K^{\frac{p}{p-1}}d\sigma\right)^{\frac{n(p-1)}{p}},\]
  \item for $p=0:$
  \[\mathcal{B}_0^{\varphi}(K)=V(K)^{1-n}\exp\left(\int_{\Sp} \log (\varphi f_K)d\theta\right)^n\left(\int_{\Sp}\frac{1}{\varphi}d\sigma\right)^n,\]
\end{itemize}
where $d\theta=\frac{\frac{1}{\varphi}}{\int_{\Sp}\frac{1}{\varphi}d\sigma}d\sigma.$

For $K\in \mathcal{F}$, define
\[
\Omega_p^{\varphi}(K)=\begin{cases}
\int_{\Sp}\varphi^{\frac{1}{p-1}}f_K^{\frac{p}{p-1}}d\sigma, & p\in [-n,\infty)\setminus\{0,1\},\\
\exp\left(\frac{\int_{\Sp}\frac{1}{\varphi}\log f_Kd\sigma}{\frac{1}{n}\int_{\Sp}\frac{1}{\varphi}d\sigma}\right), & p=0.
\end{cases}
\]
\end{definition}
A straightforward calculation shows that
\begin{align}\label{key identity}
\mathcal{B}_p^{\varphi}(\Lambda_p^{\varphi}K)=n^{n}\left(\frac{V(K)}{V(\Lambda_p^{\varphi}K)}\right)^{n-1}\mathcal{A}_p^{\varphi}(K).
\end{align}
For $L\in \mathcal{F}$ and $x\in \operatorname{int}L$, by the H\"{o}lder and Jensen inequalities
\begin{align}\label{holder}
\begin{cases}
\mathcal{B}_p^{\varphi}(L)\leq n^{n}\mathcal{A}_p^{\varphi}(L-x), & p\in [-n,1)\\
 \mathcal{B}_p^{\varphi}(L)\geq n^{n}\mathcal{A}_p^{\varphi}(L-x), & p>1.
\end{cases}
\end{align}

From now onward, we only focus on the case $p\in (-n,1)$ and establish the desired properties mentioned in the introduction. For $p>1$, the second inequality in (\ref{holder}) is in the wrong direction and hence $\Lambda_p^{\varphi}$ does not exhibit the same behavior as $\Lambda$ does.
\begin{lemma}\label{lem1}
Suppose Assumption \ref{assumption} holds and $p<1.$ We have the following.
\begin{enumerate}
  \item $\mathcal{A}_p^{\varphi}(K)\leq \left(\frac{V(\Lambda_p^{\varphi}K)}{V(K)}\right)^{n-1}\mathcal{A}_p^{\varphi}(\OP K)\leq \mathcal{A}_p^{\varphi}(\OP K).$
  \item If $p\neq 0$, then \[\Omega_p^{\varphi}(K)^{\frac{n(p-1)}{p(n-1)}}\leq \Omega_p^{\varphi}(\OP K)^{\frac{n(p-1)}{p(n-1)}}.\] If $p=0$, then
  \[\Omega_0^{\varphi}(K)\leq \Omega_0^{\varphi}(\Lambda_0^{\varphi} K).\]
  \item If $p\neq 0$, then \[c_{p}^{\varphi}\Omega_p^{\varphi}(K)^{\frac{n(p-1)}{p(n-1)}}\leq V((\OP)^i K)\leq V(K).\] If $p=0,$ then
  \[c_0^{\varphi}\Omega_0^{\varphi}(K)^{\frac{1}{n-1}}\leq V((\Lambda_0^{\varphi})^i K)\leq V(K).\]
\end{enumerate}
\end{lemma}
\begin{proof}
Inequalities of (1) and (2) follow from (\ref{key}), (\ref{key identity}), and (\ref{holder}) applied to $L=K$ and $L=\OP K.$

In view of \cite{AGN}*{Theorem 9.2}, for each $L\in \mathcal{K},$ there exists $e_p\in \operatorname{int}L$ such that
\[\mathcal{A}_p^{1}(L-e_p)\leq \mathcal{A}_p^{1}(B).\]
Therefore, for $p\neq 0$, due to (\ref{holder}) we see
\[\mathcal{B}_p^{\varphi}(L)\leq c_{p,\varphi}\quad\textrm{for any convex body}~L\in \mathcal{F}.\]
In particular, for $p\neq 0$, owing to (2), this last inequality yields
\[c_{p,\varphi}V(\OP K)\geq \Omega_p^{\varphi}(\OP K)^{\frac{n(p-1)}{p(n-1)}}\geq \Omega_p^{\varphi}(K)^{\frac{n(p-1)}{p(n-1)}}.\]
Now, (3) follows by induction.

The proof for the case $p=0$ is similar and it follows from the following inequality. By (\ref{holder}), the Jensen and Blaschke--Santal\'{o} inequality, for any convex body $L\in \mathcal{F}$ we have
\begin{align*}
\mathcal{B}_0^{\varphi}(L)\leq n^{n}\mathcal{A}_0^{\varphi}(L-s)&\leq n^nV(L)\frac{\int_{\Sp}\frac{1}{\varphi h^n_{L^s}}d\sigma}{\int_{\Sp}\frac{1}{\varphi}d\sigma}\leq c_{\varphi},
\end{align*}
where $s$ is the Santal\'{o} point of $L.$
\end{proof}
\section{Passing to a limit}
In this section, we give the proof of Theorem \ref{main theorem}. First we consider the case $p\neq -n.$ By Lemma \ref{lem1}, the operator $\OP$ satisfies the three principals mentioned in the introduction. Therefore, by \cite{Ivaki2016}*{Theorem 7.4}, there are constants $a,b$, such that
\begin{align}\label{C0 estimate}
a\leq h_{(\OP)^i K}\leq b\quad \forall i.
\end{align}
Due to the monotonicity of $\mathcal{A}_p^{\varphi}$ under $\OP$,
\[\lim_{i\to \infty}\mathcal{A}_p^{\varphi}((\OP)^iK)~\mbox{exists and is positive}\]
Thus, in view of (\ref{key}) and
\[\mathcal{A}_p^{\varphi}((\OP)^iK)\leq \left(\frac{V((\OP)^{i+1}K)}{V((\OP)^{i}K)}\right)^{n-1}\mathcal{A}_p^{\varphi}((\OP)^{i+1}K),\]
we arrive at
\[\lim_{i\to\infty}\frac{V((\OP)^{i+1}K)}{V((\OP)^{i}K)}=1.\]
By (\ref{C0 estimate}) and the Blaschke selection theorem, for a subsequence $\{i_j\}$:
\[\lim_{j\to\infty}(\OP)^{i_j}K=L\in \mathcal{K}_o,\]
From the continuity of $\OP$ (cf. \cite{Ivaki2016}*{Theorem 7.6}), it follows that
\[\lim_{j\to\infty}(\OP)^{i_j+1}K=\OP L. \]
Consequently, we must have
\begin{align*}
V(L)=\lim_{j\to\infty}V((\OP)^{i_j}K)=\lim_{j\to\infty}V((\OP)^{i_j+1}K)=V(\OP L).
\end{align*}
Now, the equality case of (\ref{key}) implies that
$\OP L=L$ and hence,
\begin{align*}\label{limit}
\varphi h_L^{1-p}f_{L}=\frac{V(L)}{\frac{1}{n}\int_{\Sp}\frac{h_L^p}{\varphi}d\sigma}.
\end{align*}
Regarding the case $-n<p<1$ and $\varphi \equiv 1$, first note that due to the result of \cite{BCD} the limiting shapes are origin-centered balls. To show that in fact they are the same ball, note that due to the monotonicity of the volume under $\Op$, the limits have the same volume.
%and $\mathcal{A}_p^1$ under $\Op$  (cf. Lemma \ref{lem1}), the limits have the same volume and the same $\mathcal{A}_p^1.$ Hence the r.h.s. of (\ref{limit}) is always the same.

Finally, regarding the third claim of Theorem \ref{main theorem}, $p=-n$, recall that $\{V(\Lambda^i K)\}$ is uniformly bounded above and below. For each $i$, by Petty \cite{Petty61} (see also \cite{Thompson}*{Theorem 5.5.14}), we can find $\ell_i\in SL(n)$ such that $\ell_i\Lambda^i K$ is in a minimal position, that is, its surface area is minimal among its volume-preserving affine transformations. Therefore, for a subsequence $\{i_j\}$, we have
\[\lim_{j\to\infty}\ell_{i_j}\Lambda^{i_j} K=L\in\mathcal{K}_o,\]
and $L$ has its Santal\'{o} point at the origin (in fact, this follows from $s(\ell_{i_j}\Lambda^{i_j} K)=\ell_{i_j}s(\Lambda^{i_j} K)=o$ and that $s$ is a continuous map with respect to the Hausdorff distance).
In particular, by \cite{Lutwak90}*{(7.12)} and the continuity of the curvature image operator we obtain
 \[\lim_{j\to\infty}\ell_{i_j}\Lambda^{i_j+1} K=\lim_{j\to\infty}\Lambda(\ell_{i_j}\Lambda^{i_j} K)= \Lambda L.\]
Meanwhile by the monotonicity of the volume product and its upper bound due to the Blaschke-Santal\'{o} inequality, we have
\[V(L)=\lim_{j\to\infty}V(\ell_{i_j}\Lambda^{i_j}K)=\lim_{j\to\infty}V(\ell_{i_j}\Lambda^{i_j+1}K)=V(\Lambda L).\]
Therefore, $\Lambda L=L$. By \cite{MP}, $L$ is an origin-centered ellipsoid. Since this ellipsoid is in a minimal position, it has to be a ball. The limit is independent of the subsequences as in the case (2) of the theorem. The proof of the theorem is finished.
%To get the stronger convergence without $\{\ell_i\}$ , note that we have shown
%\[\ell_i\Lambda^iK\to \hat{B},\quad \ell_i\Lambda^{i+k}K\to \hat{B}~\textrm{as}~i\to\infty.\]
%Therefore,
%\[(\ell_i\circ \ell_{i+k}^{-1})\circ\ell_{i+k}\Lambda^{i+k}K\to \hat{B}~\textrm{as}~i\to\infty.\]
%This in turn implies that
%\[\ell_i\circ \ell_{i+k}^{-1}\to \emph{I}.\]
%The proof of the theorem is complete.
\section{Questions}
\begin{enumerate}
  \item Let $p>1$ and $\varphi\in C^+_e(\mathbb{S}^{n-1}).$ It would be of interest to find a curvature operator whose iterations applied to any $K\in\mathcal{K}_e$ converge to the solution of the $L_p$ Minkowski problem with the prescribed even data $\varphi$.
  \item Is the limit in Theorem \ref{main theorem} independent of the subsequence?
\end{enumerate}
\bibliographystyle{amsplain}

\end{document}